\documentclass[preprint,12pt]{elsarticle}

\usepackage[T1]{fontenc}
\usepackage{lmodern}
\usepackage{amsmath,amssymb,amsthm,mathtools}
\usepackage{microtype}
\usepackage[hidelinks,bookmarksnumbered=true]{hyperref}
\journal{European Journal of Combinatorics}
\biboptions{numbers,sort&compress}
\hypersetup{
  pdftitle={Excluding a line from gammoids},
  pdfauthor={Zhen Chen}
}

\DeclareMathOperator{\cl}{cl}
\DeclareMathOperator{\si}{si}
\newcommand{\cL}{\mathcal{L}}
\numberwithin{equation}{section}
\theoremstyle{plain}
\newtheorem{thm}{Theorem}[section]
\newtheorem{lem}[thm]{Lemma}
\newtheorem{prop}[thm]{Proposition}
\newtheorem{cor}[thm]{Corollary}

\begin{document}

\begin{frontmatter}

\title{Excluding a line from gammoids}
\author[1]{Zhen Chen}
\ead{chenzzhen@126.com}

\author[1]{Zhuo Li\corref{cor1}}
\ead{lzhuo@stu.xmu.edu.cn}

\cortext[cor1]{Corresponding author.}

\address[1]{School of Mathematical Sciences,
Xiamen University, P. R. China}
\begin{abstract}
For all positive integers $\ell$ and $r$, we prove that a finite simple
rank-$r$ gammoid with no $U_{2,\ell+2}$-minor has at most $\ell(r-1)+1$
elements. This gives an affirmative answer to Problem~7.2 of Boretsky
and Walsh [European J. Combin.\ 137
(2026), 104419], for both gammoids and transversal matroids. The bound is
sharp for every $\ell$ and $r$. The key estimate is that, for every finite simple gammoid $M$, the sum of $|L|-2$ over its long lines is at most $|E(M)|-r(M)$.
\end{abstract}


\end{frontmatter}

\section{Introduction}\label{sec:introduction}

A basic problem in extremal matroid theory is to determine the largest
number of elements of a simple matroid in a given
class. Excluding a rank-two uniform minor provides a natural setting
for this problem. For integers $\ell\geq 2$ and $r\geq 1$, Kung
\cite{Kun93} proved that every simple rank-$r$ matroid with
no $U_{2,\ell+2}$-minor has at most $(\ell^r-1)/(\ell-1)$ elements.
For each fixed integer $\ell\geq 2$, let $q$ be the largest
prime power at most $\ell$. Geelen and Nelson \cite{GN10}
proved that, for all sufficiently large $r$, the maximum
number of elements of a simple rank-$r$ matroid with no
$U_{2,\ell+2}$-minor is $\frac{q^r-1}{q-1}.$
For complex-representable matroids, Geelen, Nelson, and Walsh
\cite{GNW24} proved that the maximum is
$(\ell-1)\binom{r}{2}+r$, again for sufficiently large rank.

Gammoids are matroids defined by vertex-disjoint directed paths.
They form a minor-closed class containing all transversal
matroids \cite{IP73,Mas72}. Boretsky and Walsh \cite[Theorems~1.2 and~1.3]{BW26} proved the sharp
linear bound $\ell(r-1)+1$ for simple rank-$r$ positroids with no
$U_{2,\ell+2}$-minor, and obtained the same extremal value for several
related classes, including lattice path, multi-path, bicircular, and
colaminar matroids. They then asked whether this bound also holds for
transversal matroids or, more generally, for gammoids
\cite[Problem~7.2]{BW26}.
We answer this question affirmatively.

\begin{thm}\label{thm:main}
Let $\ell,r\geq 1$ be integers. If $M$ is a simple rank-$r$ gammoid
with no $U_{2,\ell+2}$-minor, then
\begin{equation}\label{eq:main-bound}
  |E(M)|\leq \ell(r-1)+1.
\end{equation}
\end{thm}

Since every transversal matroid is a gammoid, we obtain the following
consequence.

\begin{cor}\label{cor:transversal}
Let $\ell,r\geq 1$ be integers. Every simple rank-$r$ transversal
matroid with no $U_{2,\ell+2}$-minor has at most $\ell(r-1)+1$ elements.
\end{cor}

Both bounds are sharp for every $\ell,r\geq 1$. Indeed, Boretsky and
Walsh \cite[Corollary~6.11]{BW26} give simple rank-$r$ lattice path
matroids with no $U_{2,\ell+2}$-minor and exactly $\ell(r-1)+1$
elements. Lattice path matroids are transversal
\cite{BW26}, so these examples establish sharpness in
both classes.

Our main tool is the following inequality. A \emph{long line} of a simple matroid is a rank-two flat
with at least three elements; write $\cL(M)$ for the set of long lines
of $M$.

\begin{prop}\label{prop:excess}
For every simple gammoid $M$,
\begin{equation}\label{eq:excess-bound}
  \sum_{L\in\cL(M)}(|L|-2)\leq |E(M)|-r(M).
\end{equation}
\end{prop}

Section~\ref{sec:preliminaries} introduces our notation and
recalls the required facts about gammoids and transversal
presentations.
Section~\ref{sec:excess} proves Proposition~\ref{prop:excess}
using simple strict-gammoid extensions and a counting argument
in their transversal duals.
Finally, Section~\ref{sec:main-proof} combines this estimate
with a contraction identity and induction on rank to prove
Theorem~\ref{thm:main} and Corollary~\ref{cor:transversal}.

\section{Definitions and notation}\label{sec:preliminaries}

All matroids and directed graphs in this paper are finite. We use
standard matroid terminology and refer to Oxley \cite{Oxl11} for
background.

\subsection{Matroids}\label{subsec:matroids}

For a matroid $M$, we write $E(M)$ for its ground set and $r_M$ for
its rank function, and set $r(M)=r_M(E(M))$. A set $I\subseteq E(M)$
is \emph{independent} if $r_M(I)=|I|$. A \emph{basis} is a maximal
independent set, and a \emph{circuit} is a minimal dependent set.
The \emph{nullity} of $X\subseteq E(M)$ is $|X|-r_M(X)$.

The closure of $X\subseteq E(M)$ is
\[
  \cl_M(X)=\{e\in E(M):r_M(X\cup\{e\})=r_M(X)\}.
\]
A set $F$ is a \emph{flat} if $\cl_M(F)=F$. An element $e$ is a
\emph{loop} if $r_M(\{e\})=0$, and a \emph{coloop} if it belongs to
every basis. Distinct nonloops $e$ and $f$ are \emph{parallel} if
$r_M(\{e,f\})=1$. A matroid is \emph{simple} if it has no loops or
parallel pairs. A flat $F$ is \emph{cyclic} if $M|F$ has no coloops,
or equivalently if $F$ is a union of circuits.

We write $M|S$ for restriction to $S$ and
$M\setminus S=M|(E(M)\setminus S)$ for deletion. Restriction preserves
the rank of every subset of $S$, and
\begin{equation}\label{eq:restriction-closure}
  \cl_{M|S}(X)=\cl_M(X)\cap S\qquad (X\subseteq S).
\end{equation}
Contraction of $S$ is denoted by $M/S$ and has rank function
\begin{equation}\label{eq:contraction-rank}
  r_{M/S}(X)=r_M(X\cup S)-r_M(S)
  \qquad (X\subseteq E(M)\setminus S).
\end{equation}
A matroid $N$ is an \emph{extension} of a matroid $M$ if
$E(M)\subseteq E(N)$ and $N|E(M)=M$.

We abbreviate $M/\{e\}$ and $M\setminus\{e\}$ to $M/e$ and
$M\setminus e$. A \emph{minor} is obtained by a sequence of deletions
and contractions. The dual matroid $M^*$ has as its bases the
complements of the bases of $M$, and
\begin{equation}\label{eq:dual-rank}
  r_{M^*}(X)=|X|+r_M(E(M)\setminus X)-r(M)
  \qquad (X\subseteq E(M)).
\end{equation}

A simplification $\si(M)$ is obtained by deleting all loops and keeping
one element from each nonloop parallel class. Its cardinality and rank
are independent of the chosen representatives. For a simple matroid
$M$ of rank at least two and an element $e\in E(M)$, put
\begin{equation}\label{eq:loss-definition}
  \Delta_M(e)=|E(M)|-|E(\si(M/e))|.
\end{equation}
Thus $\Delta_M(e)$ counts the elements lost under contraction of $e$
followed by simplification.

A \emph{line} is a rank-two flat. For integers $0\leq k\leq m$, the uniform matroid $U_{k,m}$ is the matroid on an $m$-element ground set whose independent sets are precisely the subsets of size at most $k$. In particular, the restriction of a simple
matroid to a line $L$ is $U_{2,|L|}$. For a simple matroid $M$, we
write $\cL(M)$ for its lines of size at least three and define its
\emph{long-line excess} by
\begin{equation}\label{eq:excess-definition}
  W(M)=\sum_{L\in\cL(M)}(|L|-2).
\end{equation}
In this notation,
Proposition~\ref{prop:excess} states that $W(M)\leq |E(M)|-r(M)$.

\subsection{Gammoids and transversal presentations}\label{subsec:presentations}

Let $D$ be a directed graph, and let $T,E\subseteq V(D)$. A set
$I\subseteq E$ is \emph{linkable to $T$} if there are pairwise
vertex-disjoint directed paths whose initial vertices are exactly
$I$ and whose terminal vertices lie in $T$. Paths of length zero
are allowed. The sets of $E$ linkable to $T$ are the independent sets
of a \emph{gammoid}. If $E=V(D)$, the gammoid is \emph{strict}.
Consequently, every gammoid is a restriction of a strict gammoid:
keep the graph and the target set, and take all vertices as the
ground set.

An indexed family $\mathcal P=(P_1,\ldots,P_b)$ of subsets of $E$
presents a \emph{transversal matroid} if its independent sets are
precisely the sets $I\subseteq E$ admitting an injection
$f:I\longrightarrow[b]$ with $x\in P_{f(x)}$ for every $x\in I$.
Here $[b]=\{1,\ldots,b\}$, with $[0]=\varnothing$. Members of the
presentation may repeat; repeated members have distinct indices.
Every transversal matroid is a gammoid: introduce new vertices
$t_1,\ldots,t_b$ and an arc $x\to t_i$ whenever $x\in P_i$, use $E$
as the ground set, and take $\{t_1,\ldots,t_b\}$ as the target set.
The resulting linkages are precisely matchings to presentation indices.

We recall that gammoids are closed under taking minors, and that
a matroid is a strict gammoid if and only if its dual is
transversal \cite{IP73,Mas72}.
Transversal matroids are closed under deletion: if
$(P_1,\ldots,P_b)$ presents $T$ and $X\subseteq E(T)$, then
$(P_1\setminus X,\ldots,P_b\setminus X)$ presents $T\setminus X$.

Let $\mathcal{P}=(P_1,\ldots,P_b)$ be a presentation of a
transversal matroid $T$.
We call $\mathcal{P}$ \emph{rank-sized} if $b=r(T)$.
It is \emph{maximal} if no $P_i$ can be replaced by a proper
superset in $E(T)$, with all other members fixed, without
changing the presented matroid.

We use the following consequence of the presentation theorem
of Brualdi and Dinolt; see \cite[Proposition~3.7]{FO22}.

\begin{lem}\label{lem:maximal-rank-sized-presentation}
Every transversal matroid has a maximal rank-sized presentation.
\end{lem}

Thus, for a strict gammoid $N$, we may choose a maximal
presentation of $N^*$ with exactly $|E(N)|-r(N)$ members.

\subsection{The alpha invariant}\label{subsec:alpha}

For a matroid $K$ and a set $X\subseteq E(K)$, Mason's \emph{alpha invariant}
is defined recursively by
\begin{equation}\label{eq:alpha-definition}
  \alpha_K(X)=|X|-r_K(X)
  -\sum_{\substack{F\subsetneq X\\F\text{ a flat of }K}}\alpha_K(F).
\end{equation}
In particular, $\alpha_K(\varnothing)=0$.

\begin{lem}[\cite{BKM11}]\label{lem:noncyclic-alpha}
If $F$ is a noncyclic flat of a matroid $K$, then
$\alpha_K(F)=0$.
\end{lem}

Consequently, the sum in \eqref{eq:alpha-definition}
may be taken over cyclic flats of $K$ properly contained in $X$. 
We also use the following description of multiplicities in a maximal
presentation; see \cite[Lemma~2.9]{Tof25}.

\begin{thm}\label{thm:multiplicities}
Let $Q$ be a strict gammoid, and let $(P_1,\ldots,P_b)$ be a maximal
rank-sized presentation of $Q^*$. For every cyclic flat $F$ of $Q$,
\[
  \bigl|\{i\in[b]:P_i=F\}\bigr|=\alpha_Q(F).
\]
\end{thm}

The equality counts presentation indices, including repeated members.

\section{The long-line excess bound}\label{sec:excess}

In this section we prove Proposition~\ref{prop:excess}.
We first show that every simple gammoid is a restriction of
a simple strict gammoid. We then bound the long-line excess
of such a restriction by its own nullity.

We use the following standard fact; see
\cite[Section~3.2, Exercise~12(a)]{Oxl11}.

\begin{lem}\label{lem:parallel-deletion}
Let $Q$ be a strict gammoid, and let $x,y$ be distinct
parallel elements of $Q$. Then $Q\setminus y$ is a strict gammoid.
\end{lem}

\begin{lem}\label{lem:simple-extension}
Every simple gammoid is a restriction of a simple strict gammoid.
\end{lem}

\begin{proof}
Let $M$ be a simple gammoid on $S$. Its defining directed graph gives
a strict gammoid $N$ with $N|S=M$. We simplify this extension while
retaining every element of $S$.

Any loop $z$ of $N$ lies outside $S$. Since $N^*$ is transversal and
$z$ is a coloop of $N^*$, we have $N^*/z=N^*\setminus z$.
Deletion preserves transversality, so
$(N\setminus z)^*=N^*/z$ is transversal. Thus deleting $z$ preserves
strictness and the restriction to $S$. Delete all such loops.

The resulting extension is loopless. If it has a parallel pair,
at least one element of the pair lies outside $S$, since $N|S$ is simple. Delete such a member using
Lemma~\ref{lem:parallel-deletion}. Repeating this step eventually
removes all parallel pairs. Each deletion preserves strictness and the
restriction to $S$, so the final extension is simple and strict.
\end{proof}

We now bound the long-line excess of a restriction of a
simple strict gammoid by the nullity of the restriction.
\begin{lem}\label{lem:intrinsic-excess}
Let $N$ be a simple strict gammoid, let $S\subseteq E(N)$, and put
$M=N|S$ and $A=E(N)\setminus S$. 
Then
\begin{equation}\label{eq:intrinsic-excess}
  W(M)\leq |S|-r(M).
\end{equation}
\end{lem}

\begin{proof}
For each $L\in\cL(M)$, define
\[
  \widehat L=\cl_N(L),\qquad c_L=|\widehat L\cap A|.
\]
Set
\[
  J=\sum_{L\in\cL(M)}c_L,
  \qquad
  A_2=A\cap\bigcup_{L\in\cL(M)}\widehat L.
\]
Write $R=r(N)$, $r=r(M)$, and $a=|A|$. The dual $T=N^*$ is
transversal of rank $b=|S|+a-R.$
By lemma~\ref{lem:maximal-rank-sized-presentation}, we choose a maximal rank-sized presentation $(P_1,\ldots,P_b)$ of $T$.

For each $L\in\cL(M)$, equation~\eqref{eq:restriction-closure} gives
\begin{equation}\label{eq:lift-intersection}
  \widehat L\cap S=\cl_N(L)\cap S=\cl_M(L)=L.
\end{equation}
Thus the sets $\widehat L=\cl_N(L)$, indexed by $L\in\cL(M)$, are pairwise
distinct. For each $L\in\cL(M)$, the set $\widehat L$ is a
rank-two flat of $N$ containing at least three elements.
As $N$ is simple, $N|\widehat L$ is a rank-two uniform matroid
with no coloops. Thus $\widehat L$ is a cyclic flat of $N$.
The flats of $N$ properly contained in $\widehat L$ are the
empty set and the singleton subsets of $\widehat L$, all of
which have alpha invariant zero. Therefore
\begin{equation}\label{eq:lift-alpha}
  \alpha_N(\widehat L)=|\widehat L|-2=|L|-2+c_L.
\end{equation}

Define
\[
  \mathcal H
  =\{i\in[b]:P_i=\widehat L
      \text{ for some }L\in\cL(M)\}.
\]
Since the sets $\widehat L$ are pairwise distinct,
Theorem~\ref{thm:multiplicities} and \eqref{eq:lift-alpha} give
\begin{equation}\label{eq:lift-indices}
  |\mathcal H|
  =\sum_{L\in\cL(M)}\alpha_N(\widehat L)
  =W(M)+J.
\end{equation}

Put $X=A\setminus A_2$, and take a basis $I$ of $T|X$.
The presentation of $T$ supplies an injection
$f:I\longrightarrow[b]$ with $x\in P_{f(x)}$ for every $x\in I$.
By the definition of $A_2$, we have
$X\cap\widehat L=\varnothing$ for every $L\in\cL(M)$.
Hence $X\cap P_i=\varnothing$ for every $i\in\mathcal H$,
so $f(I)\subseteq[b]\setminus\mathcal H$. Consequently,
\begin{equation}\label{eq:complement-matching}
  r_T(X)=|I|\leq b-|\mathcal H|=b-W(M)-J.
\end{equation}

For every $L\in\cL(M)$, the inclusion $L\subseteq S$ gives
$\widehat L=\cl_N(L)\subseteq\cl_N(S)$.
Thus $A_2\subseteq\cl_N(S)$ and $r_N(S\cup A_2)=r$.
Since $E(N)\setminus X=S\cup A_2$, the dual rank formula
\eqref{eq:dual-rank} gives
\begin{equation}\label{eq:complement-rank}
\begin{aligned}
  r_T(X)
  &=|X|+r_N(E(N)\setminus X)-R\\
  &=a-|A_2|+r-R.
\end{aligned}
\end{equation}
Substituting \eqref{eq:complement-rank} and $b=|S|+a-R$ into
\eqref{eq:complement-matching} yields
\[
  a-|A_2|+r-R
  \leq |S|+a-R-W(M)-J.
\]
Rearranging, we obtain
\[
  W(M)+J-|A_2|\leq |S|-r(M).
\]
Finally,
\[
  |A_2|
  =\left|\bigcup_{L\in\cL(M)}(A\cap\widehat L)\right|
  \leq\sum_{L\in\cL(M)}|A\cap\widehat L|
  =J.
\]
Thus \eqref{eq:intrinsic-excess} follows.
\end{proof}

\begin{proof}[Proof of Proposition~\ref{prop:excess}]
By Lemma~\ref{lem:simple-extension}, there is a simple strict gammoid
$N$ with $N|E(M)=M$. Apply Lemma~\ref{lem:intrinsic-excess} with
$S=E(M)$ to obtain $W(M)\leq |E(M)|-r(M)$.
\end{proof}

\section{Proof of the excluded-line bound}\label{sec:main-proof}

We first express the loss under contraction and simplification in terms
of the long lines through the contracted element.

\begin{lem}\label{lem:contraction-loss}
Let $M$ be a simple matroid of rank at least two, and let $e\in E(M)$.
Then
\begin{equation}\label{eq:contraction-loss}
  \Delta_M(e)-1
  =\sum_{\substack{L\in\cL(M)\\e\in L}}(|L|-2).
\end{equation}
\end{lem}

\begin{proof}
Simplicity of $M$ implies that $M/e$ is loopless: a loop $x$ of $M/e$
would satisfy $r_M(\{e,x\})=1$. For distinct $x,y\in E(M)\setminus\{e\}$,
the contraction rank formula~\eqref{eq:contraction-rank} gives
\[
  r_{M/e}(\{x,y\})=r_M(\{e,x,y\})-1.
\]
Thus $x$ and $y$ are parallel in $M/e$ exactly when $e,x,y$ lie on a
common line of $M$. It follows that the sets $L\setminus\{e\}$, as $L$
ranges over all lines through $e$, are precisely the parallel classes
of $M/e$. They partition $E(M)\setminus\{e\}$.

Contraction removes $e$, and simplification removes $|L|-2$ elements
from each class $L\setminus\{e\}$. Summing these losses gives the
identity, since lines of size two contribute zero.
\end{proof}

\begin{proof}[Proof of Theorem~\ref{thm:main}]
Fix $\ell\geq 1$ and induct on $r$. A simple rank-one matroid has
exactly one element, so the result holds when $r=1$.
Assume that $r\geq 2$ and that the result holds in rank $r-1$.
Suppose, for a contradiction, that a simple rank-$r$ gammoid $Q$ with
no $U_{2,\ell+2}$-minor violates \eqref{eq:main-bound}.

Put $n=\ell(r-1)+2$. Since $r\leq n\leq |E(Q)|$, we may
choose an $n$-element subset $S\subseteq E(Q)$ containing
a basis of $Q$. Then $M=Q|S$ is a simple rank-$r$ gammoid
with no $U_{2,\ell+2}$-minor, and
\begin{equation}\label{eq:normalized-size}
  |E(M)|=n=\ell(r-1)+2.
\end{equation}

For each $e\in E(M)$, the matroid $\si(M/e)$ is a simple rank-$(r-1)$
gammoid with no $U_{2,\ell+2}$-minor. Indeed, it is a minor of $M$,
and contraction of the nonloop $e$ lowers rank by one, while
simplification preserves rank. The induction hypothesis therefore gives
\[
  |E(\si(M/e))|\leq\ell(r-2)+1.
\]
Together with \eqref{eq:normalized-size}, this implies
\begin{equation}\label{eq:large-loss}
  \Delta_M(e)\geq\ell+1\qquad\text{for every }e\in E(M).
\end{equation}

Each long line $L$ of $M$ has at most $\ell+1$ elements; otherwise the
restriction to any $\ell+2$ elements of $L$ would be $U_{2,\ell+2}$.
Summing \eqref{eq:contraction-loss} over all elements counts each long
line once for each of its elements. Hence \eqref{eq:large-loss} and
Proposition~\ref{prop:excess} yield
\begin{equation}\label{eq:double-count}
\begin{aligned}
  \ell n
  &\leq \sum_{e\in E(M)}\bigl(\Delta_M(e)-1\bigr)\\
  &=\sum_{L\in\cL(M)}|L|(|L|-2)\\
  &\leq (\ell+1)W(M)\\
  &\leq (\ell+1)(n-r).
\end{aligned}
\end{equation}
The first and last terms imply $n\geq(\ell+1)r$. On the other hand,
\eqref{eq:normalized-size} gives
\[
  (\ell+1)r-n=r+\ell-2>0,
\]
because $r\geq 2$ and $\ell\geq 1$. This contradiction completes the
induction.
\end{proof}

\begin{proof}[Proof of Corollary~\ref{cor:transversal}]
The result follows from Theorem~\ref{thm:main}, since every
transversal matroid is a gammoid.
\end{proof}

\end{document}